\documentclass{amsart}
\usepackage[margin=3cm]{geometry}

\usepackage{graphicx}
\usepackage{xcolor} 
\usepackage{hyperref}
\theoremstyle{plain} 
\newtheorem{thm}{Theorem}[section] 
 
\newtheorem{lem}[thm]{Lemma}

\newtheorem{defn}[thm]{Definition}

\numberwithin{equation}{section}

\newcommand{\R}{\ensuremath{\mathbb{R}}}
\newcommand{\N}{\ensuremath{\mathbb{N}}}
\newcommand{\eee}{\ensuremath{\varepsilon}}

\newcommand{\ig}{\ensuremath{ \frac{\sin \pi s}{ \pi}}}
\newcommand{\igs}{\ensuremath{ \frac{1}{\Gamma(1-s)}}}
\newcommand{\al}{\ensuremath{ &\;}}

\newcommand{\lr}[1]{\left( #1 \right)}
\newcommand{\lrq}[1]{\left[#1 \right]}
\newcommand{\eqlab}[1]{\begin{equation}  \begin{aligned}#1 \end{aligned}\end{equation}}
\newcommand{\bgs}[1]{\begin{equation*} \begin{aligned}#1\end{aligned}\end{equation*}}
  \newcommand{\sys}[2][]{\begin{equation*}#1  \left\{\begin{aligned}#2\end{aligned}\right.\end{equation*}}

\begin{document}
\title[Local density of Caputo-stationary functions]{Local density of Caputo-stationary functions in the space of smooth functions}
\thanks{I am greatly indebted to Professor Enrico Valdinoci for his guidance and his precious help. I also sincerely thank Milosz Krupski for his observations, and the reviewers for their  very useful suggestions.}
\author{Claudia Bucur}
\email{claudia.bucur@unimi.it}
\address{Claudia Bucur: Dipartimento di Matematica\\ Universit\`a degli Studi di Milano \\ Via Cesare Saldini, 50 \\ 20100, Milano-Italy}
\date{}
\begin{abstract} We consider the Caputo fractional derivative and say that a function is Caputo-stationary if its Caputo derivative is zero. We then prove that any $C^k\big([0,1]\big)$ function can be approximated in $[0,1]$ by a function that is Caputo-stationary in $[0,1]$, with initial point $a<0$. Otherwise said, Caputo-stationary functions are dense in $C^k_{loc}(\R)$. 
\end{abstract}
%
%
\maketitle
\tableofcontents
\section*{Introduction}

The interest in fractional calculus has increased in the last decades given its numerous applications in viscoelasticity, signal processing, anomalous diffusion, biology, geomorphology, materials science, fractals and so on. Nevertheless, fractional calculus is a classical argument, studied since the end of the seventeenth century by many great mathematicians like Leibniz (perhaps he was the first to mention it in a letter to L'H\^{o}pital), Euler, Lagrange, Laplace, Lacroix, Fourier, Abel, Liouville, Heaviside, Weyl, Hadamard, Riemann and so on (see \cite{MR93} for an interesting time-line history). 

One can find several definitions of fractional derivatives in the literature, just to name a few, the Riemann-Liouville, the Caputo, the Riesz, the Hadamard fractional derivative, or the generalization given by the Erdélyi-Kober operator (see \cite{KST06}, \cite{MR93} and \cite{SK93} for more details on fractional integrals, derivatives and applications). The spotlight in this paper is the Caputo derivative, introduced by Michele Caputo in \cite{C67} in the late sixties. 

The Caputo fractional derivative is a so-called nonlocal operator, that models long-range interactions. For instance, if we think of a function depending on time, the Caputo fractional derivative would represent a memory effect, pointing out that the state of a system at a given time depends on past events.  In other words, the Caputo derivative describes a causal system (also known as a non-anticipative system).   

This nonlocal character of the Caputo derivative gives rise to a peculiar behavior: on a bounded interval, say $[0,1]$, one can find a Caputo-stationary function ``close enough'' to any smooth function, without any geometrical constraints. This is a surprising result when one thinks of the rigidity of the classical derivatives. For instance, the functions with null first derivative are constant functions,  the functions with null second derivatives are affine functions. Such functions cannot approximate locally any given $C^k$ function, for any fixed $k\in \N_0$.

Let $a \in \R$ and $s \in (0,1)$ be two arbitrary parameters. We define the functional space \textcolor{black}{
\eqlab{	 \label{ca1s} C_a^{1,s}  := \Big\{ f   \colon \R \to \R \mbox{ s.t. for any } x>a,  \;  f \in AC\big([a,x]\big) 
	\mbox{ and } \displaystyle & f'(\cdot){(x-\cdot)^{-s}} \in L^1\big( (a, x)\big)  \Big\} .}}
We denote here by $AC(I)$ the space of absolutely continuous functions on $I$. \textcolor{black}{ Moreover, we recall the Gamma function (see Chapter 6.1 in \cite{AS64} for other details), defined for $z>0$ as
\[ \Gamma(z) := \int_0^{+\infty} t^{z-1}e^{-t}\, dt.\] } 
We define now the Caputo derivative.
\begin{defn}
The Caputo derivative of $u\in C_a^{1,s}$ with initial point $a\in \R$ at the point $x>a$ is given by
	\begin{equation} \label{caputo}
		D^s_a u(x):= \displaystyle \frac{1}{\Gamma(1-s)}\int_a^x u'(t)(x-t)^{-s}\, dt .
		\end{equation}
\end{defn}
\noindent We define a Caputo-stationary function as follows.
\begin{defn}
We say \textcolor{black}{that} $u\in C_a^{1,s}$  is Caputo-stationary with initial point $a\in \R$  \textcolor{black}{at the point} $x>a$   if
	\bgs{  \label{caph}
	 	D^s_a u(x)=0.} 
\textcolor{black}{Let $I$ be an interval such that $a\leq \inf I$.} We say \textcolor{black}{that} $u$ is Caputo-stationary with initial point $a$ in $I$ if $D_a^su(x)=0$ holds for any $x \in I$.
\end{defn}

\textcolor{black}{For $k\in \N_{0}$}, we consider $C^k\lr{[0,1]}$ to be the space of the $k$-times continuous differentiable functions on $[0,1]$, endowed with the $C^k$-norm 
\[ \|f\|_{C^k\lr{[0,1]}} =\sum_{i=0}^k \sup_{x\in [0,1]}|f^{(i)}(x)|.\] The main result that we prove here is that for any fixed $k \in \N_0$, given any $C^k\big([0,1]\big)$ function, there exists an initial point $a<0$ and a Caputo-stationary function with initial point $a$, that in $[0,1]$ is arbitrarily close (in the $C^k$ norm) to the given function. More precisely:

\begin{thm}\label{thm:thm1}
\textcolor{black}{Let $k\in \N_0$ and $s\in (0,1)$ be two arbitrary parameters .} Then for any $f \in C^k\big([0,1]\big)$ and any $\eee>0$ there exists an initial point $a<0$ and a function $u\in C^{1,s}_a $ such that 
	\[ D_a^s u(x)=0 \text{ in } [0,\infty) \]and 
	\[\| u-f\|_{C^k\big([0,1]\big)} < \eee.\]	
\end{thm}

\bigskip
\textcolor{black}{In the next lines we recall some notions and make some preliminary remarks on the Caputo derivative. }

\textcolor{black}{The reader can see Chapter 7.5 in \cite{zygmund} for the definition of absolutely continuous functions. In particular, we use the following characterization, given in Theorem 7.29 in \cite{zygmund}, that we recall in the next Theorem.
\begin{thm}\label{acrep} A function $f$ is absolutely continuous in $ [a,b]$ if and only if $f'$ exists almost everywhere in $[a,b]$, $f'$ is integrable on $[a,b]$ and
\bgs{ f(x)-f(a)=\int_a^x f'(t)\, dt, \quad a\leq x\leq b.}
\end{thm} }
  
\textcolor{black}{By convention, when we take the Caputo derivative $D_a^s$ of a function, we assume that the function is ``causal'', i.e. that it is constant on $(-\infty,a)$. In particular, we take $u(x)=u(a)$ for any $x<a$ and this, by definition \eqref{caputo}, implies that $D_a^s u(x) =0$ for $x<a$. }
 \textcolor{black}{Lastly, we recall the Beta function (see Chapter 6.2 in the book \cite{AS64} for other details) defined for $x,y >0$ as
	\eqlab {\beta(x,y) :=\int_0^1 t^{x-1} (1-t)^{y-1} \, dt \label{b01}.} 
	We also have that 
	\[ \beta(x,y)=\frac{\Gamma(x)\Gamma(y)}{\Gamma(x+y)}.\]  In particular, the next explicit result holds
	\eqlab { \label{b02} \beta(s,1-s)=\Gamma(s)\Gamma(1-s)=\frac{\pi}{\sin\pi s}.} }

\section{Strategy of the proof}\label{sbssc}

The proof is inspired from \cite{DSV14}, where a similar result is proved for the fractional Laplacian (see \cite{DNPV12} for details about this operator). Here, we have to take into account the structure of the Caputo derivative and study in detail its behavior. 

 The main idea of the proof is that one can build a Caputo-stationary function in say $I = [0,1]$ by choosing a ``good'' given function as ``boundary'' dat\textcolor{red}{um}. For the nonlocal operators, the ``boundary'' is the complement of the given interval, for example, the fractional Laplacian takes into account the entire space and the ``boundary'' is $\R \setminus I$.  On the other hand, the Caputo derivative considers only the left-side complement and this reflects in the lack of symmetry of the boundary conditions. Namely, the ``boundary'' in the equations with the Caputo derivative is $(-\infty, 0]$, with the added convention that events start at a given point, say $t_0<0$ and $f$ is constant before time $t_0$. 
 
In order the prove Theorem \ref{thm:thm1}, we use at first the Stone-Weierstrass Theorem, that we recall here. Let $k\in \N_{0}$ be a fixed arbitrary number.

\begin{thm}\label{thm:SW}
 For any $f \in C^k\big([0,1]\big)$ and any positive $\eee$ there exists a polynomial $P$ such that \[ \|f-P\|_{C^k\big([0,1]\big)} < \eee.\]
\end{thm}
 
Then, if we prove that for any polynomial $P$ there exists a Caputo-stationary function $u$ arbitrarily close to it, by using Theorem \ref{thm:SW} we would have that 
\[ \begin{split}\|u -f \|_{C^k \big([0,1]\big)}   \leq \al \|u-P\|_{C^k \big([0,1]\big)}   +\|f-P\|_{C^k \big([0,1]\big)}  <  2\eee.
	\end{split}\] 
This would conclude the proof of Theorem \ref{thm:thm1}. 

In order to have this, we claim that it suffices to prove that for any monomial 
	\[ q_m(x)= x^m \mbox{, } m\in \N\] 
and for any $\eee_m >0$ there exists a function $u_m$ that is Caputo-stationary in $[0,1]$, such that  
	\begin{equation}\label{mapp1} \|u_m-q_m\|_{C^k \big([0,1]\big)} < \eee_m.\end{equation}
Indeed, consider an arbitrary $n \in \N$ and the polynomial $\displaystyle P(x)= \sum_{m=0}^n c_m q_m(x)$. Then the function 
$\displaystyle u(x):=\sum_{m=0}^n c_m u_m(x)$ would satisfy 
	\[ \|u  - P \|_{C^k \big([0,1]\big)} \leq \sum_{m=0}^n |c_m|\, \|u_m- q_m\|_{C^k \big([0,1]\big)}  < \sum_{m=0}^n |c_m| \eee_m =\eee,\]
where one considers for any $m$ the small quantity  $\eee_m=\displaystyle \frac{\eee}{|c_m|(n+1)}$.
Also, the function $u$ is Caputo-stationary, since the Caputo derivative is linear. Hence, the function $u$ is Caputo-stationary and is ``close'' to any polynomial. This proves the claim.
\bigskip

In the rest of the paper, we prove that we can find a Caputo-stationary function close to any given monomial. To do this, we proceed as follows: \\
\begin{itemize}
\item \textcolor{black}{ In Section \ref{sectrfcp}, we obtain a representation formula for $u$, when $D_a^s u(x)= 0 $ in $(b,\infty)$ for a given $b>a$ and having prescribed $u$ on $(-\infty,b]$. To do this, we prove that having $D_a^s u(x)= 0 $ is equivalent to having a particular integro-differential equation. We then obtain a representation formula for the integro-differential equation, hence for our initial equation. } \\
\item  In  Section \ref{sectaps}, we prove that there exists a sequence $(v_j)_{j\in \N}$ of Caputo-stationary functions in $(0,\infty)$ such that, uniformly on bounded  subintervals of $(0,\infty)$, we have that $\lim_{j\to \infty} v_j(x) = \kappa x^s$, for a \textcolor{black}{suitable} constant $\kappa>0$.\\
\item In Section \ref{sectma} we prove that there exists a Caputo-stationary function with an arbitrarily large number of derivatives prescribed. We do this by taking advantage of the particular structure of the function $x^s$. If we take any derivative of such a function, say $(x^s)^{(i)}= s(s-1)\dots(s-i+1) x^{s-i} ,$ for \textcolor{black}{$x> 0$} this derivative never vanishes.\\
\item \textcolor{black}{Section \ref{sectthm1} deals with the proof of Theorem \ref{thm:thm1}. Prescribing the derivatives of $u$ such that, for $m\in \N$,  they vanish at $0$ until the order $m-1$, and are equal to $1$ at order $m$, using a Taylor expansion and performing a blow-up argument, we can conclude the proof of the main theorem.}
%
\end{itemize}

\section{A representation formula for a Caputo-stationary function}\label{sectrfcp}
The purpose of this section is to deduce a Poisson-like representation formula for a function $u\in C_a^{1,s}$ that is Caputo-stationary with initial point $a$ in the interval $(b,\infty)$ for $b>a$, and fixed outside, i.e. 
	\bgs{ &D_a^s u(x) =  0 &\text{ in } & (b,\infty),\\
	&\mbox{  prescribed data }  &\text{ in } & (-\infty,b]. }
To do this, we prove that this problem is equivalent to the integro-differential equation
	\bgs{&\int_b^x u'(t)(x-t)^{-s}\, dt = g(x) &\text{ in } & (b,\infty),\\
	&\mbox{  prescribed data }  &\text{ in } & (-\infty,b], }
for a given function $g$ (that depends on the prescribed data of the initial problem). Then, we 
introduce in Theorem \ref{thm:probc} a representation formula for this integro-differential equation. With these two results in hand, we obtain a representation for the solution of the initial problem. 
Moreover, we present here an interior regularity result.
	\begin{center}
\begin{figure}[htpb]
	\hspace{0.6cm}
	\begin{minipage}[b]{0.85\linewidth}
	\centering
	\includegraphics[width=0.90\textwidth]{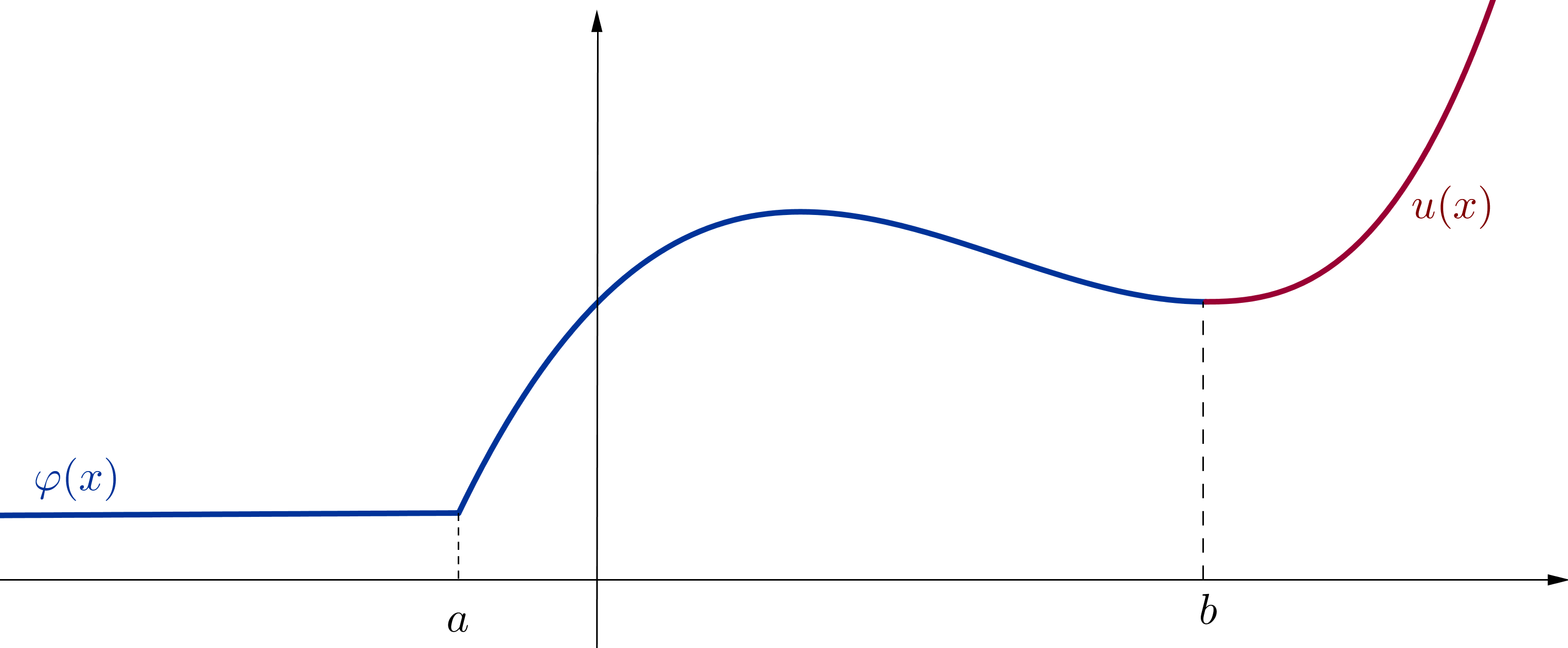}
	\caption{A Caputo-stationary function in $(b,\infty)$ prescribed on $(-\infty,b]$}   
	\label{fign:Lem31}
	\end{minipage}
\end{figure} 
\end{center} 

\textcolor{black}{In this section, we fix the arbitrary parameters $a,b \in \R$ with $b>a$ and $s\in(0,1)$.}  

We state in the next Lemma the equivalence between the two problems above.

\begin{lem}\label{lem:int11}
Let $\varphi \in C\big((-\infty,b] \big)\cap C^1\big([a,b]\big)$ such that $\varphi(x) = \varphi(a)$ in $(-\infty,a]$. Then $u\in  C_a^{1,s}$  satisfies the equation
	\bgs{ \label{intr1}
			D_a^s u(x)& =0 &\text{ in } & (b,\infty),\\
			u(x)&=\varphi(x) &\text{ in }  & (-\infty,b]}
if and only if it satisfies 
	\bgs{ \label{intr2}
		 \int_b^x u'(t)(x-t)^{-s} \, dt   &=- \int_{a}^b \varphi'(t)(x-t)^{-s}\, dt & \text{ in }& (b,\infty),\\
		u(x)&=\varphi(x)&\text{ in } & (-\infty,b]. }
\end{lem}
The reader can see a qualitative graphic of a function described by Lemma \ref{lem:int11} in Figure \ref{fign:Lem31}. \textcolor{black}{An explicit example of such a function is build in the Appendix, in Figure \ref{fign:es1}.} 
\begin{proof} Since $\varphi \in C^1\big([a,b])$ we have for any $x\geq b$  
	\bgs{\bigg |\int_a^{b} \varphi'(t)(x-t)^{-s} \, dt \bigg |\leq  \sup_{t\in [a,b]} |\varphi'(t)| \frac{ (x-a)^{1-s}-(x-b)^{1-s}}{1-s}<\infty.}
Hence the map $x\mapsto\displaystyle  \int_a^{b} \varphi'(t)(x-t)^{-s} \, dt $ is well defined \textcolor{black}{in $[b,\infty)$}. Using the definition \eqref{caputo} for $x >b$ we have that
	\begin{equation*}
		\begin{split}
		\Gamma(1-s) D_a^s u(x)
		= \al  \int_b^x u'(t)(x-t)^{-s} \, dt + \int_a^{b}  u'(t)(x-t)^{-s} \, dt \\
		=\al  \int_b^x u'(t)(x-t)^{-s} \, dt + \int_a^{b} \varphi'(t)(x-t)^{-s} \, dt  .
		\end{split}
	\end{equation*}
It follows that $D_a^s u(x)=0$ on $(b,\infty)$ is equivalent to 
		\bgs{ 
		 \int_b^x u'(t)(x-t)^{-s} \, dt   =- \int_a^{b} \varphi'(t)(x-t)^{-s} \, dt  \quad \text{ in } (b,\infty).	}
 This concludes the proof of the Lemma. 
\end{proof}

In the following Theorem we introduce a representation formula for an integro-differential equation.
\begin{thm}
\label{thm:probc}
Let $g \in C_b^{1,1-s}$. The problem 
\eqlab{ \label{probc1}
			\int_b^x u'(t)(x-t)^{-s} \,dt  & = g(x) \quad \mbox{ in } (b,\infty),\\
			 u(b)  &=  0 } admits on $[b,\infty)$ a unique solution $u\in C_b^{1,s}$. Moreover, for any $x>b$,
			\begin{equation} \label{solc1} 
	u(x)= \textcolor{black}{\ig}\int_b^x g(t)(x-t)^{s-1} \, dt .
	\end{equation}
\end{thm}

\begin{proof}
\textcolor{black}{We prove this theorem by showing that $u$ given in \eqref{solc1} is  well defined, belongs to the space $C_b^{1,s}$ and is the unique solution of the problem \eqref{probc1}.}
 
Since $g$ belongs to $C_b^{1,1-s}$ (recall \eqref{ca1s}), for any $x> b$ we have that
	\bgs{|u(x)| \leq\ig  \int_b^x |g(t)| (x-t)^{s-1} \, dt  \leq c_s\sup_{ t \in [b,x]}	|g(t)|  (x-b)^{s} <\infty,} where $c_s$ is a positive constant. Hence the definition \eqref{solc1} is well posed.  

\bigskip 
We prove that $u$ belongs to $ C_b^{1,s}$.
We claim that 
\textcolor{black}{ \eqlab{ \label{cbsu1}  g\in C_b^{1,1-s} & \mbox{ and } u \mbox{ as in } \eqref{solc1} \implies  \\
 & u \in AC\big([b,\infty)\big) \mbox{ and} \\
  & u'(y)=\frac{\sin \pi s} {\pi}  \lr{ \int_b^y g'(\tau)(y-\tau)^{s-1}\, d\tau + g(b)(y-b)^{s-1}} \quad  \mbox{   a.e. in  } [b,\infty). }}
We fix an arbitrary $x>b$. According to definition \eqref{ca1s},  $g \in AC\big( [b,x]\big) $ and \textcolor{black}{thanks to Theorem \ref{acrep} we have that}
for any $t\in [b,x]$
	\[ g(t)=\int_b^t g'(\tau)\, d\tau + g(b).\]
And so in \eqref{solc1} we have that
	\eqlab{ \label{bla2}   \frac{\pi}{\sin \pi s} \, u(x)
		= \al \int_b^x \lr{ \int_b^t g'(\tau)\, d\tau } (x-t)^{s-1}\, dt + g(b) \int_b^x (x-t)^{s-1}\,dt .}
We compute
	\eqlab {\label{bla1}  \int_b^x (x-t)^{s-1} \, dt = \frac{(x-b)^s}{s} = \int_b^x (y-b)^{s-1} \, dy.}
Tonelli theorem applied to the positive measurable function $|g'(\tau)|(x-t)^{s-1}$ on the domain 
\eqlab{ \label{rev2}D_{b,x}:=\{ (t,\tau) \text{ s.t. } b\leq t\leq x, b\leq \tau\leq t\}} 
with the product measure $d(t,\tau)$ gives 
	\eqlab{\label{rev1} \iint_{D_{b,x}} |g'(\tau)|\,   (x-t)^{s-1} \, d (t, \tau) =   \al \int_b^x |g'(\tau)| \lr{  \int_\tau^x (x-t)^{s-1} \, dt }\,d\tau \\ 
	 	=\al\frac{1}{s}\int_b^x |g'(\tau)| (x-\tau)^s \, d\tau \\ 
			\leq\al\frac{(x-b)^s}{s}  \|g'\|_{L^1\big((b,x)\big)}, }
which is a finite quantity. Hence $|g'(\tau)| (x-\tau)^{s-1} \in L^1\lr{D_{b,x}, d(t,\tau)}$ and by Fubini theorem \textcolor{black}{ and using \eqref{bla1}} it follows that
	\bgs{ \int_b^x \lr{\int_b^t g'(\tau)\, d\tau } (x-t)^{s-1}\, dt  = \al \int_b^x g'(\tau) \lr{ \int_{\tau}^x  (x-t)^{s-1} \, dt} \, d\tau\\
		= \al \int_b^x  g'(\tau) \lr{\int_\tau^x (y-\tau)^{s-1} \, dy} \, d\tau\\
			=\al \int_b^x \lr{ \int_b^y g'(\tau) (y-\tau)^{s-1} \, d\tau }\, dy.}
Inserting this and identity \eqref{bla1} into \eqref{bla2}, we obtain that
	\[  \frac{\pi}{\sin \pi s}\,  u(x) = \int_b^x \lr{  \int_b^y g'(\tau) (y-\tau)^{s-1} \, d\tau + g(b) (y-b)^{s-1} } \, dy.\] 
\textcolor{black}{Hence $u$ is the integral function of a $L^1\big((b,x)\big)$ function (thanks to \eqref{rev1}) and recalling that $u(b)=0$, according to Theorem \ref{acrep} we have that $u\in AC\big([b,x]\big)$.}
Moreover, almost everywhere in $[b,x]$ 
	 \bgs {\label{u1d}
	\frac{\pi} {\sin \pi s} \,u'(y)=\int_b^y g'(\tau)(y-\tau)^{s-1}\, d\tau + g(b)(y-b)^{s-1}. }
\textcolor{black}{With this, given the arbitrary choice of $x$, we have proved the claim \eqref{cbsu1}.}

We claim now that $ u'(\cdot) (x-\cdot)^{-s} \in L^1\big( (b,x))$. Using the second identity in \eqref{cbsu1}, we obtain that
	\eqlab{ \label{fbca} &\frac{\pi}{\sin \pi s}  \int_b^x  | u'(y)| (x-y)^{-s} \, dy \\
	\leq 
	 \al \int_b^x  \lr{\int_b^y |g'(\tau)| (y-\tau)^{s-1}  \, d\tau} (x-y)^{-s} \, dy + | g(b) |  \int_b^x (y-b)^{s-1} (x-y)^{-s}dy . }
Tonelli theorem applied to the positive function $|g'(\tau)| (y-\tau)^{s-1} (x-y)^{-s} $ on the domain $D_{b,x}$ \textcolor{black}{given in \eqref{rev2}} with the product measure $d(y,\tau)$ gives
	\bgs { \label{tt31} \iint_{D_{b,x}}  |g'(\tau)| (y-\tau)^{s-1} (x-y)^{-s} \, d(y, \tau) = \al \int_b^x  |g'(\tau)| \lr{ \int_\tau^x (y-\tau)^{s-1} (x-y)^{-s} \, dy } \, d\tau.} 
By using the change of variables $\displaystyle t = \frac{y-\tau}{x-\tau}$, thanks to definition \eqref{b01} \textcolor{black}{and identity \eqref{b02}} we have that
 \eqlab{ \label{bfcomp} 
\int_\tau^x (y-\tau)^{s-1} (x-y)^{-s} \, dy  = \int_0^1 t^{s-1}(1-t)^{-s} \, dt = \frac{\pi}{\sin\pi s}.}
Hence we obtain that
	\eqlab{  \label{Fubsto}  \iint_{D_{b,x}}  |g'(\tau)| (y-\tau)^{s-1} (x-y)^{-s} \, d(y ,\tau) 
 = \al  \frac{\pi}{\sin\pi s} \| g'\|_{L^1\big((b,x)\big)}.} 
\textcolor{black}{From this and using again \eqref{bfcomp} with $b=\tau$,} we obtain in \eqref{fbca} that 
	\bgs{ \al  \int_b^x  | u'(y)| (x-y)^{-s} \, dy \leq   \|g'\|_{ L^1 \big((b,x)\big)} + |g(b)|.}
Hence  $ u'(\cdot) (x-\cdot)^{-s} \in L^1\big( (b,x))$, as claimed. \textcolor{black}{From this and \eqref{cbsu1}, recalling definition \eqref{ca1s} it follows that} $u$ 
belongs to the space $C_b^{1,s}$.

\bigskip 
We prove now that $u$ is a solution of the problem \eqref{probc1}. Using the second identity in \eqref{cbsu1} we have that 
	\eqlab{ \label{blaq1}  \frac{\pi}{\sin \pi s} \int_b^x u' (y)(x-y)^{-s}\,  dy  
						=\al \int_b^x \lr { \int_b^y g'(\tau)(y-\tau)^{s-1} \, d\tau } (x-y)^{-s} \, dy  \\ \al +  g(b) \int_b^x (y-b)^{s-1} (x-y)^{-s}\, dy  .}
Thanks to \eqref{Fubsto}, we have that $|g'(\tau)| (y-\tau)^{s-1} (x-y)^{-s} \in L^1 (D_{b,x}, d(y,\tau)) $. We apply Fubini theorem and using \eqref{bfcomp} we get that  
	\bgs{ \int_b^x \lr{\int_b^y g'(\tau)(y-\tau)^{s-1} (x-y)^{-s}\, d\tau }\, dy =\al  \int_b^x g'(\tau) \lr{ \int_\tau^x (y-\tau)^{s-1} (x-y)^{-s} \, dy }  \, d\tau, \\
		=\al \frac{\pi}{\sin \pi s}  \lr{ g(x)-g(b) } .} 
Thanks again to \eqref{bfcomp}, in \eqref{blaq1} it follows that 	
\bgs{ \int_b^x \al u'(y)(x-y)^{-s}\, dy  = 	g(x),}
therefore $u$ is a solution of the problem \eqref{probc1}.
\bigskip

The solution is unique. We prove this by taking two different solutions $u_1,u_2\in C_b^{1,s}$ of the problem \eqref{probc1}. Let $u:=u_1-u_2$, then $u$ satisfies
	\bgs{ \int_b^x u'(t)(x-t)^{-s}\, dt &=0 &\text{in} \quad &(b,\infty),	\\
		u(b)&=0 .&&}
We take any $y>x$, we multiply both terms by the positive quantity $(y-x)^{s-1}$, integrate from $b$ to $y$ and obtain that
	\eqlab {\label{bla3}  \int_b^y \lr{ \int_b^x u'(t) (x-t)^{-s}\, dt } (y-x)^{s-1} \, dx =0.}
Since $u \in  C_b^{1,s}$, \textcolor{black}{we use Tonelli theorem on $D_{b,y}$ (we recall definition \eqref{rev2}) and by \eqref{bfcomp} we obtain that}
 	\bgs{  \iint_{D_{b,y}}  |u'(t)|  (x-t)^{-s}(y-x)^{s-1} \, d(x,t) 
	=\al  \int_b^y |u'(t)| \lr{ \int_t^y  (x-t)^{-s}(y-x)^{s-1} \, dx} \, dt\\
	=\al  \frac{\pi}{\sin \pi s} \|u'\|_{L^1\big((b,y)\big)}, } which is a finite quantity. Fubini theorem then allows us to compute 
\bgs{  \int_b^y \lr{ \int_b^x u'(t) (x-t)^{-s}\, dt } (y-x)^{s-1} \, dx =&\;  \int_b^y  u'(t) \lr{  \int_t^y (x-t)^{-s}(y-x)^{s-1} \, dx }\, dt \\
		= \al  \frac{\pi}{\sin \pi s}u(y)  	  .} 
It follows from \eqref{bla3} and from the initial condition $u(b)=0$ that $u_1(x)= u_2(x)$ on $[b,\infty)$. Therefore $u$ given in \eqref{solc1} is the unique solution of the problem \eqref{probc1} and this concludes the proof of the Theorem.
\end{proof}

We introduce an interior regularity result.
\begin{lem} \label{intreg}
Let $g \in C^{\infty}\big([b,\infty)\big)$ and $u$ be defined as in \eqref{solc1}. Then $u\in C^{\infty}\big((b,\infty)\big)$. 
\end{lem} 
\begin{proof}
We prove by induction that the next statement, which we call $P(n)$, holds for any $n\in \N$: 
\textcolor{black}{  \[ u\in C^n\big((b,\infty)) \]} 
and
 \eqlab{ \label{undiff1} u^{(n)}(y) = \ig\lr{  \int_b^y g^{(n)}(\tau) (y-\tau)^{s-1} \, d\tau +  \sum_{i=0}^{n-1} \tilde c_{s,i} g^{(i)}(b) (y-b)^{s-n+i}  } \\\mbox{for any } y\in (b,\infty) ,}
where 
	\begin{equation} \label{ctcsi1} \tilde c_{s,i} = \begin{cases}    (s-1)\dots (s-n+i+2) (s-n+i+1) \quad &\text{ for }  i\neq n-1\\
										  1  \quad \quad\quad\quad &\text{ for }  i= n-1. \end{cases}
	\end{equation}
		
We denote by
	\[ v(y):=\int_b^y g'(\tau)(y-\tau)^{s-1}\, d\tau. \] 
\textcolor{black}{and from \eqref{cbsu1} we have that almost anywhere in $[b,\infty)$
	\eqlab{\label{uprim} u'(y) = \ig \lr{ v(y) +   g(b)(y-b)^{s-1}}.} Since $g\in C^{\infty}\big([b,\infty)\big)$, we have in particular that $g'\in C_b^{1,1-s}$ hence from the definition of $v$ and \eqref{cbsu1} we get that $v\in AC\big([b,\infty)\big)$. It follows that $u'\in C\big((b,\infty)\big)$, since it is a sum of continuous functions. Therefore $u\in C^1\big((b,\infty)\big)$ and \eqref{uprim} holds pointwise in $(b,\infty)$}.
And so $P(1)$ is true. 

In order to prove the inductive step, we suppose that $P(n)$ holds and prove $P(n+1)$.
Let now
	\[ v(y):=\int_b^y  g^{(n)}(\tau)(y-\tau)^{s-1}\, d\tau.\] From \eqref{undiff1} we have that for any $y\in (b,\infty)$
	\eqlab{\label{unv1} u^{(n)}(y) =\ig\lr{ v(y)+  \sum_{i=0}^{n-1} \tilde c_{s,i} g^{(i)}(b) (y-b)^{s-n+i} }  .} 
\textcolor{black}{Since $g\in C^{\infty}\big([b,\infty)\big)$, in particular we have that $g^{(n)}\in C_b^{1,1-s}$ hence from the definition of $v$ and thanks to \eqref{cbsu1} we get that $v\in AC\big([b,\infty)\big)$ and 
almost everywhere on $[b,\infty) $
	\[ v'(y) = \int_b^y g^{(n+1)}(\tau) (y-\tau)^{s-1}\, d\tau + g^{(n)}(b) (y-b)^{s-1}.\] Now, also $g^{(n+1)}\in C_b^{1,1-s}$ and so, thanks to \eqref{cbsu1}, the map 
	\eqlab{\label{yg1} y \mapsto \displaystyle \int_b^y g^{(n+1)}(\tau) (y-\tau)^{s-1}\, d\tau \quad \in AC\big([b,\infty)\big) .} It yields that $v\in C^1\lr{(b,\infty)}$ and so from \eqref{unv1} we get that $u^{(n+1)}\in C\lr{(b,\infty)}$. Taking the derivative of \eqref{unv1} we have that pointwise in $(b,\infty)$}
\bgs{ \frac{\pi}{\sin\pi s}  u^{(n+1)} (y)=  \al \int_b^y g^{(n+1)}(\tau) (y-\tau)^{s-1}\, d\tau  +  	
%
	g^{(n)}(b)(y-b)^{s-1}+ 
	\sum_{i=0}^{n-1} \tilde c_{s,i} g^{(i)}(b) (s-n+i) (y-b)^{s-n+i-1} \\
		=\al \int_b^y g^{(n+1)}(\tau)(y-\tau)^{s-1}\, d\tau  +  \sum_{i=0}^{n} \tilde c_{s,i} g^{(i)}(b) (y-b)^{s-n+i},}
		 where we have used \eqref{ctcsi1} in the last line. 
Therefore the statement $P(n+1)$ is true and the proof by induction is concluded. 

It finally yields that $u\in C^{\infty}\lr{(b,\infty)}$ and this concludes the proof of the Lemma.
\end{proof}
\smallskip 

\section{Existence of a sequence of Caputo-stationary functions that tends \\to the function $x^s$}\label{sectaps}

In this Section we introduce some preliminary results, on which we will base the proof of Theorem \ref{thm:thm1}. The purpose of this section is to build a sequence of functions that are Caputo-stationary  in $(0,\infty)$ and that tends uniformly on bounded subintervals of $(0,\infty)$ to the function $x^s$. We do this by building a Caputo-stationary function in $(1,\infty)$, that at the point $1+\eee$ is asymptotic to $\eee^s$ and then we use a blow-up argument.

\bigskip

We fix the arbitrary parameter $s\in (0,1)$. We introduce the first Lemma of this Section.

\begin{lem}\label{lem1}
Let $\psi_0 \in C^1\big( [0,1]\big) \cap C\big((-\infty,1]\big)$ be such that 
	\eqlab {	\label{fifi41} &\psi_0(x)=\psi_0(0) & \text{ for any } &x\in (-\infty, 0],\\
			&\psi_0(x) =0 &\text{ for any } &x\in \bigg[\frac{3}{4}, 1\bigg],\\
	&\psi'_0 (x)<0 &\text{ for any } &x\in \bigg[0,\frac{3}{4}\bigg). }
Let $\psi\in C_0^{1,s}$ be the solution of the problem
	\eqlab {\label{prob1}
			D_0^s\psi(x) &=0 &\text{ in }& (1,\infty),\\
			\psi(x)&=\psi_0(x)&\text{ in } &(-\infty,1].} \textcolor{black}{Then $\psi \in C^{\infty}\big((1,\infty)\big)$ and} if $x=1+\eee$, we have that 
	\eqlab{\label{psieee} \psi(1+\eee) =\kappa \eee^s + \mathcal O(\eee^{s+1})  }
as $\eee \to 0$, for some $\kappa>0$. 
\end{lem}

\textcolor{black}{An explicit example of a function described in Lemma \ref{lem1} is depicted in Figure \ref{fign:es2} in the Appendix.}
\begin{proof}[Proof of Lemma \ref{lem1}]

Thanks to Lemma \ref{lem:int11} we have that $\psi \in C_0^{1,s}$ is solution of the problem \eqref{prob1} if and only if 
	\bgs{ \int_1^x \psi'(t)(x-t)^{-s}\, dt &= -\int_0^{3/4} \psi_0'(t) (x-t)^{-s} \, dt &\mbox{in } &(1,\infty), \\
										\psi(x) &=\psi_0(x) &\mbox{in } &(-\infty,1].}	
On $[1,\infty)$ we define the function
\eqlab{\label{gbla1} g(x):=- \int_0^{3/4} \psi_0'(t) (x-t)^{-s} \, dt, }
\textcolor{black}{hence our problem is now
	\eqlab{ \label{psil4} \int_1^x \psi'(t)(x-t)^{-s}\, dt &=g(x) &\mbox{in } &(1,\infty), \\
										\psi(x) &=\psi_0(x) &\mbox{in } &(-\infty,1].}	
We claim that $g\in C^{\infty}\big([1,\infty)\big)$.}
For that, let $F\colon [1,\infty)\times [0,3/4]\to \R$ be defined as $F(x,t):=\psi_0'(t)(x-t)^{-s}$. 
\textcolor{black}{Now, for any $h>0$ arbitrarily small we have that
\[ \bigg| \frac{F(x+h,t)-F(x,t)}h \bigg| \leq \sup_{t\in [0,3/4]} |\psi_0'(t)| \bigg|\frac{(x+h-t)^{-s}-(x-t)^{-s}}h\bigg|. \]
Since the map $[1,\infty)\ni x \mapsto (x-t)^{-s}$ is differentiable for any $t\in [0,3/4]$, by the mean value theorem we have that for $\theta \in (0,h)$
\[ \bigg|\frac{(x+h-t)^{-s}-(x-t)^{-s}}h\bigg| \leq s (x+\theta-t)^{-s-1} \leq s(x-t)^{-s-1}.\]
Then
\[\bigg|\frac{F(x+h,t)-F(x,t)}h\bigg|\leq s \sup_{t\in [0,3/4]} |\psi_0'(t)| (x-t)^{-s-1} \in L^1\big([0,3/4],dt\big),\]
hence by the dominated convergence theorem, we can pass the limit inside the integral and obtain that
\[ g'(x)= -\int_0^{3/4} \partial_x F (x,t)\, dt =  s\int_0^{3/4} \psi'_0(t)  (x-t)^{-s-1}\, dt.\]}
We can now take for any $n \in \N$ the function $F_n\colon [1,\infty)\times [0,3/4]\to \R$ to be $F_n(x,t):=\psi_0'(t)(x-t)^{-s-n}$ and repeat the above argument. We obtain that $g$ is $C^\infty\big([1,\infty)\big)$, as claimed and
moreover for any $n\in \N_0$ we have that
	\eqlab{ \label{gn1} g^{(n)}(x) =  -\bar c_{s,n}\int_0^{3/4} \psi_0'(t) (x-t)^{-s-n}\,dt, }
where
	\begin{equation} \label{ctcns2} \bar c_{s,n} = \begin{cases} (-s)(-s-1)\dots (-s-n+1) &\mbox{ for } n\neq 0\\
																1 &\mbox{ for } n=0.\end{cases}\end{equation}
																
Since $\psi(1)=0$ and $g\in C^\infty\big([1,\infty)\big)$ (hence in particular $g\in C_1^{1,1-s}$), thanks to Theorem \ref{thm:probc} we get that the problem \eqref{psil4} admits a unique solution $\psi \in C_1^{1,s}$ given by
\eqlab{ \label{solll} &\psi(x)=\ig  \int_1^x g(t) (x-t)^{s-1} \, dt &\mbox{ in } &(1,\infty),\\
		&\psi(x)=\psi_0(x) &\mbox{ in } & (-\infty,1].} 
 \textcolor{black}{Moreover, we claim that $\psi \in C_0^{1,s}$. Indeed, from Lemma \ref{intreg} we get that $\psi\in C^{\infty}\big( (1,\infty)\big)$. Also 
$ \lim_{x\to 1^+} \psi(x) =0 =\psi(1)$ and so from this and the hypothesis we have that $\psi \in C^{\infty}\big( (1,\infty)\big)  \cap  C^1\lr{[0,1]} \cap C(\R)$, hence $\psi \in AC\lr{[0,\infty)}$. Also for any $x>0$
\[ \int_0^x |\psi'(t)(x-t)^{-s}| \, dt \leq c_s  \|\psi'\|_{L^{\infty}\lr{(0,x)}} x^{1-s} <\infty, \]
and so the claim follows from definition \eqref{ca1s}.}
Therefore, $\psi \in C_0^{1,s}$ is the unique solution of problem \eqref{psil4} and from Lemma \ref{lem:int11} it follows that \eqref{solll} is also the unique solution of problem the \eqref{prob1}.  
\bigskip

We prove now the claim \eqref{psieee}. Let $x=1+\eee$. Then from \eqref{solll} we have that 
\bgs{  \frac{\pi}{\sin \pi s} \psi(1+\eee) =  \int_1^{1+\eee} g(\tau) (1+\eee-\tau)^{s-1} \, d\tau.} 
The change of variables $z= (\tau-1)/\eee$ gives
	\bgs{ \frac{\pi}{\sin \pi s} \psi(1+\eee) = \eee^s  \int_0^1 g(\eee z+ 1) (1-z)^{s-1}\, dz.}
Using definition \eqref{gbla1} we have that 
	\[ g(\eee z+ 1) = -\int_0^{3/4} \psi_0'(t)(\eee z+1-t)^{-s} \, dt ,\] 
hence
	\bgs{ \frac{\pi}{\sin \pi s} \psi(1+\eee) = -\eee^s \int_0^1 \lr{ \int_0^{3/4} \psi_0'(t)(\eee z+1-t)^{-s}\, dt} (1-z)^{s-1}\,dz.} 
Tonelli theorem on $[0,1]\times [0,3/4]$ applied to the function $|\psi_0'(t)|(\eee z+1-t)^{-s} (1-z)^{s-1}$ yields
	\bgs{  \iint_{ [0,1]\times [0,3/4]} \al |\psi_0'(t)|(\eee z+1-t)^{-s} (1-z)^{s-1} d(t,z)\\
=\al  \int_0^{3/4} |\psi_0'(t)| \lr{ \int_0^1 (1-z)^{s-1}(\eee z+1-t)^{-s}\, dz}\, dt.}
We have that $(\eee z+1-t)^{-s}\leq (1-t)^{-s}\leq 4^s$, hence
	 \bgs{  \int_0^{3/4} |\psi_0'(t)| \lr{ \int_0^1 (1-z)^{s-1}(\eee z+1-t)^{-s}\, dz}\, dt \leq \al 4^s  \int_0^{3/4} |\psi_0'(t)| \lr{\int_0^1 (1-z)^{s-1}\, dz}\, dt  \\
			\leq \al \frac{3 \cdot 4^{s-1}}s \sup_{t\in[0,3/4]}|\psi_0'(t)|  ,}
which is finite. Therefore $|\psi_0'(t)|(\eee z+1-t)^{-s} (1-z)^{s-1}\in L^1\big([0,1]\times [0,3/4], d(t,z)\big)$ and by Fubini theorem we have that
	\eqlab{ \label{psibla2}
			 \frac{\pi}{\sin \pi s}  \psi(1+\eee)  =\al  -\eee^s \int_0^{3/4}\psi_0'(t) \lr{ \int_0^1 (\eee z+1-t)^{-s} (1-z)^{s-1}  \, dz}\, dt \\
			=\al -\eee^s \int_0^{3/4} \psi_0'(t)  I_s(\eee,t)\, dt.}
We consider the function $f(z)=(\eee z+1-t)^{-s}$ and make a Taylor expansion with a Lagrange reminder in $0$. Namely, one has that there exists $c \in (0,z)$ such that
	\[ f(z)=\sum_{i=0}^n f^{(i)}(0) \frac{z^i}{i!} + \frac{f^{(n+1)}(c)}{(n+1)!}z^{n+1}.\]
We have that for some $c\in (0,z)$ 
	\[ (\eee z+ 1-t)^{-s} = \sum_{i=0}^n  \frac{\bar c_{s,i}} {i!}\eee^i  (1-t)^{-s-i}  z^i + \frac{\bar c_{s,n+1}}{(n+1)!} \eee^{n+1} (\eee c+1-t)^{-s-n-1} z^{n+1},\]
where $\bar c_{s,i}$ is given in \eqref{ctcns2}. Using this, we have that
	\bgs{ I_s(\eee,t)= \al \sum_{i=0}^n  \frac{\bar c_{s,i}} {i!}\eee^i  (1-t)^{-s-i}\int_0^1 (1-z)^{s-1} z^i\, dz \\ \al +  \frac{\bar c_{s,n+1}}{(n+1)!} \eee^{n+1} (\eee c+1-t)^{-s-n-1} \int_0^1 (1-z)^{s-1} z^{n+1}\, dz.}
We use the definition \eqref{b01} of the Beta function and continue
	\bgs{  I_s(\eee,t) = \sum_{i=0}^n  \frac{\bar c_{s,i}\beta(i+1,s) } {i!}\eee^i  (1-t)^{-s-i} + \frac{\bar c_{s,n+1}\beta(n+2,s)}{(n+1)!} \eee^{n+1} (\eee c+1-t)^{-s-n-1}.}
	In \eqref{psibla2} we obtain that
	\eqlab { \label{psibla3}  \frac{\pi}{\sin \pi s}  \psi(1+\eee) =\al -\eee^{s} \sum_{i=0}^n \frac{\bar c_{s,i}\beta(i+1,s) } {i!} \eee^i \int_0^{3/4}\psi_0'(t) (1-t)^{-s-i}\, dt \\ \al- \eee^{s+n+1} \frac{\bar c_{s,n+1}\beta(n+2,s)}{(n+1)!}  \int_0^{3/4} \psi_0'(t)  (\eee c+1-t)^{-s-n-1} \, dt.}
We notice that $(\eee c+1-t)^{-s-n-1}\leq 4^{s+n+1}$ and it follows that
	\[ \bigg| \int_0^{3/4} \psi_0'(t)  (\eee c+1-t)^{-s-n-1} \, dt\bigg| \leq 3\cdot  4^{s+n} \sup_{t \in [0,3/4]} |\psi_0'(t)|,\] which is finite.
 We define then the finite quantities
	\bgs{ C_{s,\psi_0,i}:=\al  -\frac{\bar c_{s,i} \beta(i+1,s) }{i!} \int_0^{3/4} \psi_0'(t)(1-t)^{-s-i}\, dt\\ =\al  \frac{\beta(i+1,s) } {i!}   g^{(i)}(1) \quad  \mbox{for} \quad  i=0,\dots,n } and
	 \bgs{ C_{s,\psi_0,n+1}: = \al -\frac{\bar c_{s,n+1} \beta(n+2,s) }{(n+1)!} \int_0^{3/4} \psi_0'(t)(\eee c +1-t)^{-s-n-1}\, dt \\
=\al  \frac{\beta(n+2,s) } {(n+1)!}   g^{(n+1)}(\eee c +1), }
where we have used \eqref{gn1}.  

It follows in \eqref{psibla3} that
	\[ \frac{\pi}{\sin \pi s}\psi(1+\eee) = \sum_{i=0}^{n+1} C_{s,\psi_0,i}  \textcolor{black}{\eee^{s+i}} . \]
 This gives for $\eee \to 0$ that
	\[\psi(1+\eee)= \kappa\eee^s + \mathcal O (\eee^{s+1}),\]
where 
	\bgs{ \kappa =C_{s,\psi,0}= \beta(1,s) g(1) \textcolor{black}{=} -   \beta(1,s) \int_0^{3/4} \psi_0'(t)(1-t)^{-s} \, dt. }
Since $-\psi_0'(x)>0$ in $[0,3/4)$ by hypothesis (see \eqref{fifi41}), we have that 
\[-\int_0^{3/4} \psi_0'(t)(1-t)^{-s} \, dt>0.\] 
This implies that $\kappa$  is strictly positive and it concludes the proof of the Lemma.
\end{proof}
 \bigskip

Blowing up the function built in Lemma \ref{ls1}, we obtain a sequence of Caputo-stationary functions in $(0,\infty)$ that on $(0,\infty)$ tends to the function $x^s$.  

\begin{lem}\label{ls1}
There exists a sequence $(v_j)_ {j \in \N}$ of functions $v_j \in  C^{1,s}_{-j}\cap C^{\infty}\big((0,\infty)\big)$ such that  for any $j \in \N$
	\begin{equation} \label{pbvj1}
		\begin{aligned}
			D^s_{-j} v_j(x)&= 0 &\text{ in } &(0,\infty),	\\
			v_j(x) &= 0 &\text{ in } & \Big[-\frac{j}4,0\Big]
		\end{aligned} 
	\end{equation}  
and for any $x>0$
	 \eqlab{ \label{limvj} \lim_{j\to \infty} v_j(x)=\textcolor{black}{\kappa} x^s,  } for some $\textcolor{black}{\kappa}>0$. 
Moreover, on any bounded subinterval $I\subseteq (0,\infty)$ the convergence is uniform.
\end{lem}	
A qualitative example of a sequence described in Lemma \ref{ls1} is depicted in Figure \ref{fign:Lem42}.

\begin{center}
 \begin{figure}[htpb]
	\hspace{0.6cm}
	\begin{minipage}[b]{0.85\linewidth}
	\centering
	\includegraphics[width=0.95\textwidth]{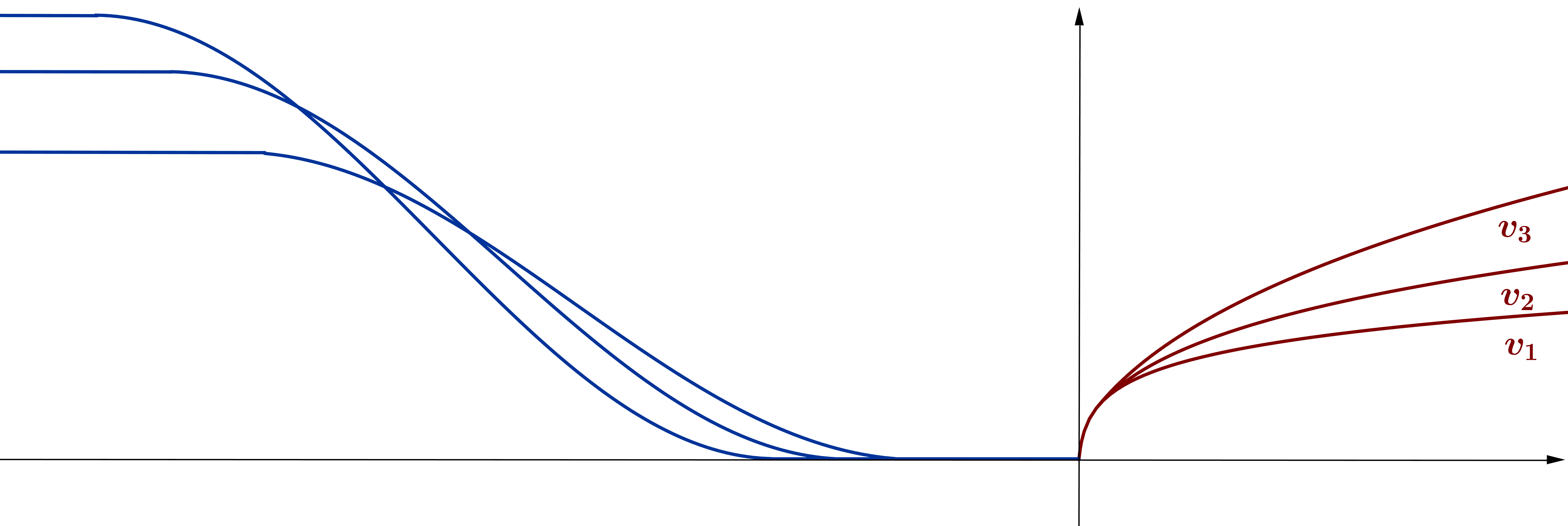}
	\caption{A sequence of Caputo-stationary functions in $(0,\infty)$}   
	\label{fign:Lem42}
	\end{minipage}
\end{figure}
\end{center} 
\begin{proof}
We consider the function $\psi$ solution of the problem \eqref{prob1} as introduced in Lemma \ref{lem1}, and define for any $j\in \N$
\[ v_j(x) := j^s \psi\bigg(\frac{x}{j} +1\bigg).\]  We prove that for any $j\in \N$ the function $v_j$ is solution of the problem \eqref{pbvj1}.

Recalling Lemma \ref{lem1}, we have that $\psi(x)=\psi_0(x)$ in $(-\infty,1]$, hence $v_j(x)=j^s \psi_0 \displaystyle \bigg(\frac{x}{j} +1\bigg)$ when $\displaystyle \frac{x}{j} +1 \leq 1$, i.e. when $x\leq 0$. Moreover, from conditions \eqref{fifi41} we have that $v_j(x)=j^s \psi_0(0)$ when $\displaystyle \frac{x}{j} +1 \leq 0$, hence when $x\leq -j$ and $v_j(x)=0$ when $\displaystyle\frac{3}{4}\leq \displaystyle \frac{x}{j} +1 \leq 1$, hence for $x\in \displaystyle \left[-\frac{j}4,0\right]$.  
According to the fact that $\psi \in C_0^{1,s}\cap C^{\infty}\big((1,\infty)\big)$, we have that $v_j\in C_{-j}^{1,s} \cap C^{\infty}\big((0,\infty)\big)$.  Furthermore, since $\psi$ is solution of the problem \eqref{prob1}, we have by the definition \eqref{caputo} that
	\bgs{	D_{-j}^s v_j(x) =\al \igs  \int^x_{-j} v_j'(t)(x-t)^{-s} \, dt\\
					= \al  \frac{ j^{s-1}}{\Gamma(1-s)} \int^x_{-j}\psi'\Big(\frac{t}{j}+1\Big) (x-t)^{-s} \, dt.}
We use the change of variables $y= t/j +1$ and obtain
	\bgs{	D_{-j}^s v_j(x) =\al  \igs \int_0^{x/j+1} \psi'(y)\Big(\frac{x}{j}+1-y\Big)^{-s}\, dy \\
					=\al D_0^s \psi \Big(\frac{x}{j} +1\Big).}
This implies that $D_{-j}^s v_j(x) =0 $ whenever  $D_0^s \psi \displaystyle \left(\frac{x}{j} +1\right) =0$. From \eqref{prob1}, this happens when $\displaystyle \frac{x}j +1 >1$, hence for $x>0$. 
\textcolor{black}{And so in conclusion} we have that for any $j\in \N$  the functions $v_j \in C_{-j}^{1,s} \cap C^{\infty}\big((0,\infty)\big)$ satisfy
	\bgs{ D_{-j}^s v_j(x) &= 0 &\mbox{ in } &  (0,\infty), \\ 
			v_j(x) &= 0 &\mbox{ in } & \left[-\frac{j}4,0\right] }
and
	\bgs{ v_j(x)&=j^s\psi_0\lr{\frac{x}j +1} 	&\mbox{ in } &  (-\infty,0],\\
				v_j(x)&=j^s\psi_0(0) 	&\mbox{ in } &  (-\infty,-j].}
In particular, $v_j$ is solution of the problem \eqref{pbvj1} for any $j \geq 1$.
\bigskip

We prove now that as $j \to \infty$, the sequence $v_j (x)$ tends on $(0,\infty)$ to the function $\kappa x^s$, for a suitable constant $\kappa>0$. Using \eqref{psieee}, for $x>0$ and for a large $j$ we have that
	\[ v_j(x) = j^s \psi \left(\frac{x}{j}+1 \right) = j^s \left( \kappa \frac{x^s}{j^s} + \mathcal O \left(\frac{x^{s+1}}{j^{s+1}}\right)\right) = \kappa x^s + \mathcal O \left(\frac{x^{s+1}}{j}\right).\]
By sending $j$ to infinity we obtain that 
	\[ \lim_{j\to \infty} v_j(x)=\kappa x^s.\]
On any bounded subinterval $I\subseteq (0,\infty)$, we have that
	\bgs{ \lim_{j \to \infty} \sup_{x\in I} |v_j(x) - \kappa x^s| = 0. } It follows also that on any bounded subinterval $I\subseteq (0,\infty)$ the sequence $v_j$ is uniformly bounded.
This concludes the proof of the Lemma. 
\end{proof}

\section{Existence of a Caputo-stationary function with arbitrarily \\large number of derivatives prescribed}\label{sectma}

Using Lemma \ref{ls1} we prove that there exists a Caputo-stationary function with arbitrarily large number of derivatives prescribed. Namely, for any $m\in \N$ we want to prove that we can find a Caputo-stationary function $v$ and a point $p$, such that the derivatives of 
$v$ in $p$ vanish until the order $m-1$. More precisely:

\begin{thm}\label{thm4}
 For any $m \in \N$ there exist a point $p>0$, a constant $R>0$ and a function $v \in C_{-R}^{1,s} \cap C^{\infty} \big((0,\infty)\big)$ such that
	\eqlab{ \label{cc1} 
				D^s_{-R} v(x)&=0 &\text{ in } &(0, \infty), \\
		v(x)&=0 &\text{ in } & \Big [-\frac{R}4,0\Big]}	
and 
	\eqlab{ \label{cc2}
			&v^{(l)}(p)=0 & & \text{ for any } \quad  l< m\\
		&v^{(m)}(p)=1.&&}
	\end{thm}

\begin{proof}

We consider $\mathcal Z$ to be the set of the pairs $(v,x)$ of all functions $v\in C_{-R}^{1,s}\cap C^{\infty} \big((0,\infty)\big)$ satisfying conditions \eqref{cc1} for some $R>0$, and $x\in (0,\infty)$. More precisely
	\textcolor{black}{\bgs{\mathcal Z = \Big\{ (v,x) \text{ s.t. } x\in (0,\infty) \mbox{ and } \exists\, R>0 \text{ s.t. } & v\in C_{-R}^{1,s} \cap C^{\infty} \big((0,\infty)\big), D^s_{-R} v=0 \text{ in } (0, \infty), v =0 \text{ in } \Big [-\frac{R}4,0\Big]  \Big\}. }}
We fix $m\in \N$. To each pair $(v,x)\in \mathcal Z$ we associate the vector $ \big(v(x), v'(x), \dots, v^{(m)}(x)\big) \in \R^{m+1}$ and consider $\mathcal V$ to be the vector space spanned by this construction. We claim that this vector space exhausts $\R^{m+1}$. 
Suppose by contradiction that this is not so and $\mathcal V$ lays in a hyperplane. Then there exists a vector $(c_0,c_1,\dots,c_m)\in \R^{m+1}\setminus \{0\}$ orthogonal to any vector $  \big(v(x), v'(x), \dots, v^{(m)}(x)\big) $ with $(v,x) \in \mathcal Z$, hence 
	\[ \sum_{i=0}^m c_i v^{(i)}(x) = 0.\] 
We notice that for any $j\geq 1$ the pairs $(v_j,x)$ with $v_j$ satisfying problem \eqref{pbvj1} and $x\in (0,\infty)$ belong to the set $\mathcal Z$. It follows that for any $j\geq 1$ we have that
	\eqlab{  \label{bla10} \sum_{i=0}^m c_i v_j^{(i)} (x) =0.}

Let $\varphi \in  C^{\infty}_c\big((0,\infty)\big)$. Integrating by parts we have that for any $i\in \N$
	\bgs{ \int_\R v_j^{(i)}(x)\varphi(x)\, dx= (-1)^i \int_\R v_j(x) \varphi^{(i)}(x)\, dx.} 
Thanks to Lemma \ref{ls1}, the sequence $v_j$ is uniformly convergent to $\kappa x^s$ on any bounded subinterval $I\subseteq (0,\infty)$, for some $\kappa>0$. By the  dominated convergence theorem we have that
	\[ \lim_{j \to \infty}  \int_\R v_j^{(i)}(x)\varphi(x)\, dx = (-1)^i\lim_{j \to \infty}  \int_{\R} v_j(x) \varphi^{(i)}(x)\, dx =   (-1)^i \int_{\R} \kappa x^s \varphi^{(i)}(x) \, dx.\] We integrate by parts one more time and obtain that
	\[  (-1)^i \int_\R \kappa x^s \varphi^{(i)}(x)\, dx = \int_\R \kappa (x^s)^{(i)} \varphi(x)\, dx.\]
It follows that
	\[ \lim_{j \to \infty}  \int_\R v_j^{(i)}(x)\varphi(x)\, dx =   \int_\R  \kappa ( x^s)^{(i)} \varphi(x)\, dx .\] Multiplying by $c_i$ and summing up, we obtain that
\bgs{  \lim_{j\to \infty}  \int_\R \sum_{i=0}^m c_i v_j^{(i)}(x) \varphi(x)\, dx  =  \int_\R  \sum_{i=0}^m c_i \kappa ( x^s)^{(i)} \varphi(x)\, dx .}
From this and equality \eqref{bla10} we finally obtain that 
	\[  0 = \int_{\R}\sum_{i=0}^m c_i \kappa (x^s)^{(i)} \varphi(x)\, dx \] for any $\varphi \in C_c^{\infty}\big((0,\infty)\big)$. 
This implies that on $(0,\infty)$
	\[ 0=  \kappa \sum_{i=0}^m c_i(x^s)^{(i)} =\kappa \sum_{i=0}^m c_i s(s-1)\dots(s-i+1)x^{s-i}.   \]
We divide this relation by $\kappa$ (that is strictly positive) and multiply by $x^{m-s}$ and obtain that for any $x \in (0,\infty) $
	\[ \sum_{i=0}^m c_i s(s-1)\dots(s-i+1) x^{m-i}  =0.\]
We have here a polynomial that vanishes for any positive $x$. Thanks to the fact the $s \in(0,1)$ the product $s(s-1) \dots (s-i+1)$ is never zero, therefore one must have $c_i= 0 $ for every $i\in\N_0$. This is a contradiction since the vector $(c_o,\dots,c_m)$ was assumed not null. Hence the vector space $\mathcal V$ exhausts $\R^{m+1}$ and there exists $(v,p) \in \mathcal Z $ such that $\big( v(p), v'(p),\dots, v^{(m)}(p)\big)=(0,0,\dots,1)$. This concludes the proof of Theorem \ref{thm4}.
\end{proof}

\section{Proof of Theorem \ref{thm:thm1}}\label{sectthm1}
This section is dedicated to the proof of Theorem \ref{thm:thm1}. We translate and rescale the function $v$ as given in Theorem \ref{thm4}. The derivatives of the rescaled function vanish in $0$ until the order $m-1$, and the $\mbox{m}^{th}$ derivative equals $1$. Using a Taylor expansion, we obtain that this rescaled function well approximates the monomial $q(x)=c_m x^m$.

\begin{proof}[Proof of Theorem \ref{thm:thm1}]	
In Section \ref{sbssc} we explained why it suffices to prove that for any $m\in \N$ and any monomial $q_m(x)=x^m$ there exists a Caputo-stationary function $u$ such that 
	\[ \|u-q_m\|_{C^k\big([0,1]\big)} < \eee.\] 
For an arbitrary $m\in \N$, we take for convenience the monomial \[ q_m(x)=\frac{x^{m} }{m!}. \] \textcolor{black}{Also, we consider $p,R>0$ and the function $v$ as introduced in Theorem \ref{thm4}} and 
we translate and rescale $v$. Let $\delta $ be a positive quantity (to be taken conveniently small  in the sequel) and let $u$ be the function
		 \[ u(x):= \frac{v(\delta x +p)}{\delta^m} .\]
Since $v\in C_{-R}^{1,s} \cap C^{\infty}\big((0,\infty)\big)$ we have that $u \in  C_{ \frac{-p-R}{\delta}}^{1,s} \cap C^{\infty} \lr{\Big(-\displaystyle \frac{p}{\delta}, \infty\Big)}$ and 
	\bgs{  \Gamma(1-s) D_{ \frac{-p-R}{\delta} }^s u(x) =\al    \int_{  \frac{-p-R}{\delta} }^x u'(t)(x-t)^{-s}\, dt \\
		= \al  \delta^{1-m} \int_{  \frac{-p-R}{\delta} }^x v'(\delta t+ p) (x-t)^{-s}\, dt.}
We change the variable $y=\delta t +p$ and obtain that
	\bgs{ \Gamma(1-s)  D^s_{\frac{-p-R}{\delta} } u(x)  = \al  \delta^{s-m} \int_{-R}^{\delta x +p} v'(y) (\delta x+ p-y)^{-s}\, dy\\=\al 
\Gamma(1-s)  D^s_{-R} v(\delta x +p).}
Let  $a:=\displaystyle \frac{-p-R}{\delta}$. Using the properties \eqref{cc1} of $v$ we obtain that
	\bgs{ D_a ^su(x) =0 \text{ in } \Big(-\frac{p}{\delta}, \infty\Big). }  
With this notation, we have that $u\in C_a^{1,s}$ and since $\displaystyle -\frac{p}{\delta}<0$, that $D_a ^su(x) =0 \text{ in } [0, \infty). $
 
Furthermore, from the conditions \eqref{cc2} and the definition of $u$ we get that
	\bgs{ u^{(l)}(0)& = \delta^{l-m}v^{(l)}(p)= 0 & & \text{ for any } \quad  l< m\\
		u^{(m)}(0)& = v^{(m)}(p)= 1.&&}
Let for any $x>-p/\delta$ \[ g(x):= u(x) -q_m(x).\] We have that
	\eqlab{ \label{gr1} g^{(l)}(0)&= 0  &\mbox{ for any } &l\leq m\quad  \mbox{and}\\
		g^{(m+l)} (x)&= u^{(m+l)}(x) & \mbox{ for any } &l\geq 1 .} 
Moreover for $l\geq 1 $ we have that $u^{(m+l)}(x)= \delta^l v^{(m+l)} (\delta x +p)$ and it follows that
	\[| g^{(m+l)}(x) | = \delta^{l} |v^{(m+l)}(\delta x+p) | .\]
Hence for $x \in [0,1]$ we have the bound
	\eqlab{ \label{bg1} | g^{(m+l)}(x) |\leq \delta^l \sup_{y \in [p, p+\delta] } |v^{(m+l)} (y)| = \tilde C \delta^l,} where $\tilde C$ is a positive constant.
We consider the derivative of order $k$ of $g$ and take its Taylor expansion with the Lagrange reminder. \textcolor{black}{Thanks to \eqref{gr1}, for some $c\in (0,x)$ we have that
	\[ g^{(k)}(x)= \sum_{i=\max\{ k,m+1\}} ^{k+m+1} g^{(i)}(0) \frac{x^{i-k}}{(i-k)!} +g^{(m+k+2)} (c)  \frac{x^{m+2}}{(m+2)!} .\] }
Using \eqref{bg1} for any $x\in [0,1]$, eventually renaming the constants we have that 
\textcolor{black}{\[ |  g^{(k)}(x)| \leq C \sum_{i={\max\{1,k-m\}}}^{k+2} \delta^i, \]}
therefore for $k\in \N_{0}$
\[ |  g^{(k)}(x)| = |q_m^{(k)}(x) -u^{(k)}(x)|= \mathcal O(\delta) .  \]
If we let $\delta\to 0$ we have that $u^{(k)}$ approximates $q_m^{(k)}$. Finally, for any small $\eee(\delta)>0$
	\[ \|u-q_m\|_{C^k\big([0,1]\big)} < \eee\] and this concludes the proof of Theorem \ref{thm:thm1}.
\end{proof}

\section*{Appendix}
In this Appendix, we want to give some explicit examples related to some Lemmas that were introduced in this paper.

At this purpose, to give an example of Lemma \ref{lem:int11}, we take $a=0, b=1, s=1/2$ and the function $\varphi(x)=x$ in $[0,1]$ and $\varphi(x)=0$ in $(-\infty,0)$. We built the function $u\in C_0^{1,1/2}$ that satisfies
	\eqlab{\label{es1} D_0^{\frac{1}2} u(x)& =0 &\text{ in } & (1,\infty),\\
			u(x)&=x &\text{ in }  & [0,1],\\
			u(x)&=0&\text{ in }  & (-\infty,0) .}
Let \[ g(x):= -\int_0^1 \frac{\varphi'(t)}{\sqrt{x-t}} \, dt=-\int_0^1 (x-t)^{\frac{1}2}\, dt = 2\sqrt{x-1}-2\sqrt{x}.\] 
According to Lemma \ref{lem:int11} and to Theorem \ref{thm:probc}, the unique solution of the problem \eqref{es1} is given by 
\[ u(x) =u(1)+  \frac{1}{\pi}\int_1^x \frac{g(t)}{\sqrt{x-t}}\, dt,\]
and computing, this gives
\[ u(x)=\frac{2}{\pi} \lr{x\arcsin \frac{1}{\sqrt x}-\sqrt{x-1}}. \]
We depict this function in the following Figure \ref{fign:es1}.
\begin{center}
\begin{figure}[htpb]
	\hspace{0.6cm}
	\begin{minipage}[b]{0.85\linewidth}
	\centering
	\includegraphics[width=0.90\textwidth]{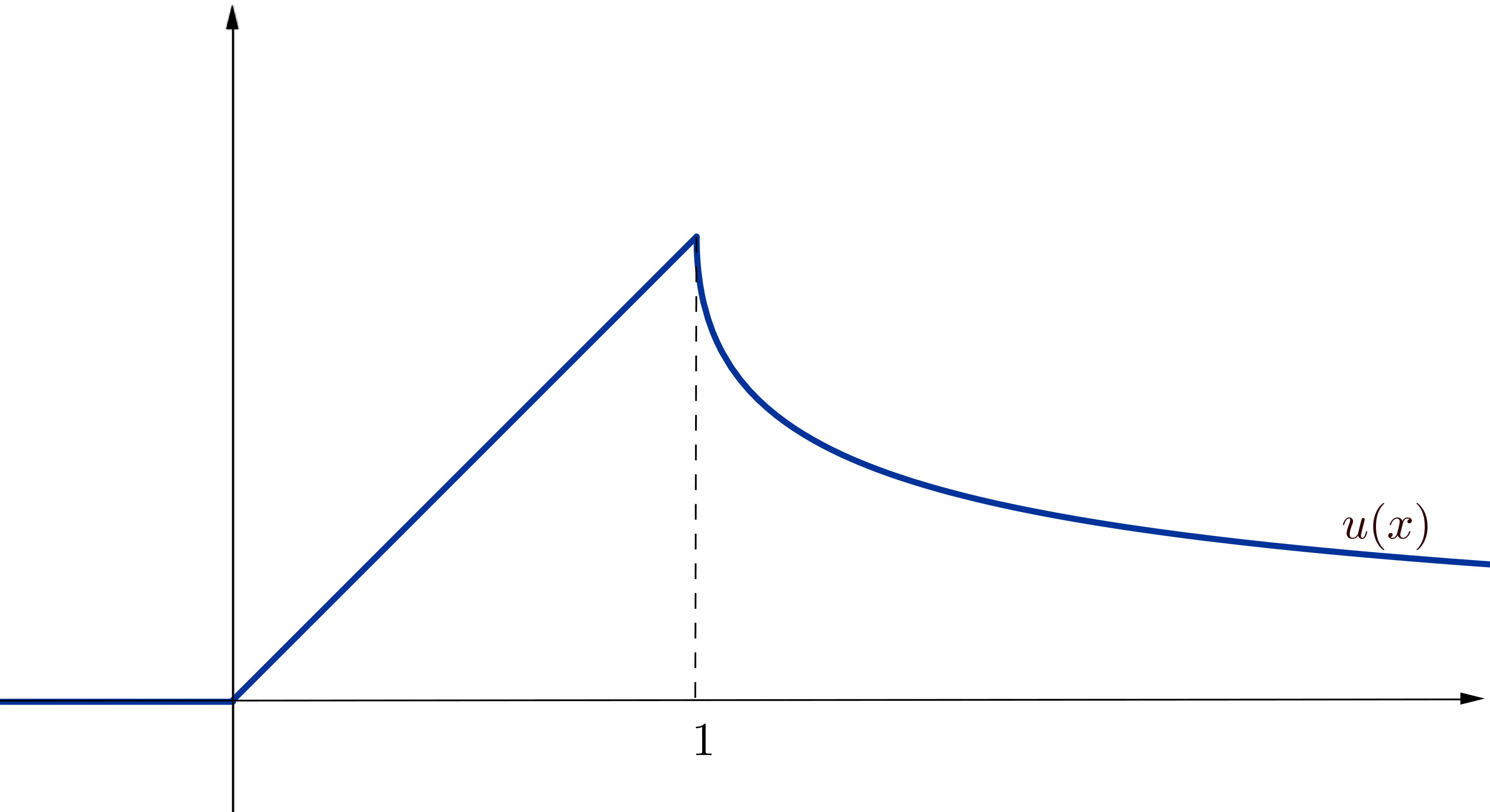}
	\caption{A Caputo-stationary function in $(1,\infty)$ prescribed on $(-\infty,1]$}   
	\label{fign:es1}
	\end{minipage}
\end{figure} 
\end{center}

In Lemma \ref{lem1}, we take $a=0, b=1, s=1/2$ and the quadratic function
\sys [\psi_0(x) =]{ &\frac{16}9 \lr{x-\frac{3}4}^2 &\mbox{ in } &\lrq{0,\frac{3}4},\\
                       	& 0  & \mbox{ in } &\lrq{\frac{3}4,1}.}
                So we are looking for a function $\psi  \in C_0^{1,1/2}$ that satisfies   
                \eqlab{\label{es2} D_0^{\frac{1}2} \psi(x)& =0 &\text{ in } & (1,\infty),\\
			\psi(x)&=\psi_0(x) &\text{ in }   & (-\infty,1] .} 
			The solution, according again to Lemma \ref{lem:int11} and to Theorem \ref{thm:probc} is given by
			\[\psi(x) = \frac{1}{\pi}\int_1^{x}g(t)(x-t)^{-\frac{1}2} \, dt, \quad \mbox{ where } \quad g(t)= -\int_0^{\frac{3}4}  \psi_0'(t) (x-t)^{-\frac{1}2}\, dt.\]
Computing this, we have that
\[ g(t)= -\frac{16}{27} \lr{ 8t^{\frac{3}2}-9t^{\frac{1}2} -(4t-3)^{\frac{3}2}}\] 
and  
\[ \psi(x) =\frac{1}{27\pi} \lrq{ 27 \pi + \sqrt{x-1} (-48x+52) +\arcsin \frac{1}{\sqrt{x}} (96x^2-144x) -\arcsin \frac{1}{\sqrt{4x-3}} (96x^2-144x+54) }.\]
We depict this function in the following Figure \ref{fign:es2}.	
\begin{center}
\begin{figure}[htpb]
	\hspace{0.6cm}
	\begin{minipage}[b]{0.85\linewidth}
	\centering
	\includegraphics[width=0.90\textwidth]{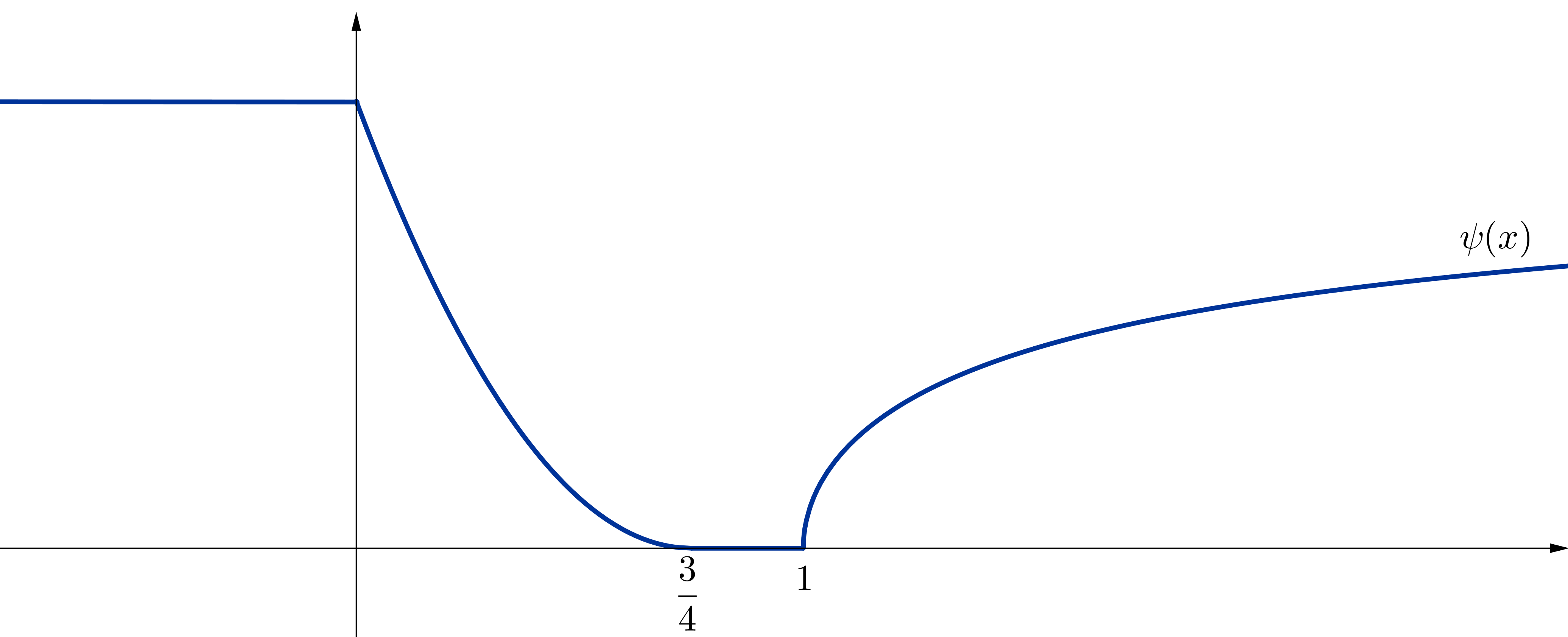}
	\caption{A Caputo-stationary function in $(1,\infty)$ prescribed on $(-\infty,1]$}   
	\label{fign:es2}
	\end{minipage}
\end{figure} 
\end{center}	

\bibliography{biblio}
\bibliographystyle{plain}

\end{document}